\documentclass[a4paper,12pt]{article}

\usepackage{amsmath}
\usepackage{amssymb}
\usepackage{amsfonts}
\usepackage{amsthm}
\usepackage{mathrsfs}
\usepackage{graphicx}
\usepackage{makeidx}
\usepackage{subcaption}
\usepackage{color}
\usepackage{mdframed}
\usepackage{hyperref}

\usepackage[left=3cm,right=3cm,top=3cm,bottom=3cm]{geometry}

\newtheorem{theorem}{Theorem}

\newtheorem{proposition}{Proposition}

\newtheorem{remark}{Remark}

\newtheorem{example}{Example}

\usepackage{amsopn}

\DeclareMathOperator{\N}{\mathbb{N}}
\DeclareMathOperator{\C}{\mathbb{C}}
\DeclareMathOperator{\R}{\mathbb{R}}
\DeclareMathOperator{\Z}{\mathbb{Z}}

\newcommand{\scal}[2]{\langle{#1},{#2}\rangle} 

\newcommand{\hd}[1]{{\color{black}{#1}}}
\DeclareMathOperator{\ba}{\mathbf a}

\title{\bf Solving moment and polynomial optimization problems on Sobolev spaces\footnote{This work was co-funded by the European Union under the project ROBOPROX (reg. no. CZ.02.01.01{\slash}00{\slash}22{\_}008{\slash}0004590).}}

\begin{document}

\author{Didier Henrion$^{1,2}$, Alessandro Rudi}

\footnotetext[1]{CNRS, LAAS, Universit\'e de Toulouse, France. }
\footnotetext[2]{Faculty of Electrical Engineering, Czech Technical University in Prague, Czechia.}
\footnotetext[3]{Inria, Ecole Normale Sup\'erieure, CNRS, PSL Research University, Paris, France. }

\date{Draft of \today}

\maketitle

\begin{abstract}
Using standard tools of harmonic analysis, we state and solve the problem of moments for \hd{non-negative} measures supported on the unit ball of a Sobolev space of multivariate periodic trigonometric functions. We describe outer and inner semidefinite approximations of the cone of Sobolev moments. They are the basic components of an infinite-dimensional moment-sums of squares hierarchy, allowing to \hd{numerically solve} non-convex polynomial optimization problems on infinite-dimensional Sobolev spaces with global convergence guarantees.\\[1em]
\end{abstract}

\section{Introduction}

The {\em moment-SOS hierarchy}, also known as the Lasserre hierarchy, was originally introduced in the early 2000s to solve globally finite-dimensional {\em polynomial optimization problems} (POP) \cite{l10,l09,bpt13,n23}. Then it was extended to polynomial differential equations and their optimal control, see \cite{hkl20} for a recent overview of applications and more references. The main technical ingredients on which the moment-SOS hierarchy relies are sums of squares (SOS) representations of positive polynomials (the so-called Positivstellens\"atze) \cite{m08} and its dual problem of moments \cite{s17} providing conditions satisfied by moments of a \hd{non-negative} measure supported on a finite-dimensional set. These conditions are truncated to finite degrees, yielding a converging hierarchy of semidefinite optimization problems of increasing size that can be solved numerically using interior-point algorithms \cite{nn94,bn01}.

After more than two decades of research in polynomial optimization, the application range of the moment-SOS hierarchy is now being extended to challenging nonconvex nonlinear optimization problems formulated on infinite-dimensional functional spaces, e.g. problems of calculus of variations or partial differential equations. The moment-SOS hierarchy has been recently extended to reproducible kernel Hilbert spaces for polynomial optimization \cite{rmb20} or optimal transport \cite{mvbvr21}.
Measures supported on infinite-dimensional spaces arise naturally as relaxed controls for infinite-dimensional optimization \cite{f99}. Ambrosio's superposition principle \cite{a08} whose finite-dimensional Euclidean version was used in \cite{hk14} to prove convergence of the moment-SOS hierarchy for approximating the region of attraction of polynomial differential equations, has been extended to infinite-dimensional Hilbert or Banach spaces \cite{at17}. Measures on infinite-dimensional spaces are also used in fluid dynamics, see e.g. \cite{fmw21} or more recently \cite{rt22} which presents itself as an infinite-dimensional extension of the finite-dimensional SOS setup of \cite{tgd18}. The solution of the moment problem for measures supported on infinite-dimensional spaces is more technically involved than its finite-dimensional counterpart, see \cite{ikr14,s18} and references therein. The recent reference \cite{hikv23} shows however that the heat equation with polynomial nonlinearities can be solved numerically with the infinite-dimensional moment-SOS hierarchy, with convergence guarantees provided by a recent solution of the moment problem on nuclear spaces \cite{ikkm23}.

The present paper aims at contributing to the numerical solution of the {\em infinite-dimensional} moment problem in a \hd{specific} functional analytic framework\hd{,} which makes its analysis as well as its numerical implementation as simple as possible. We use basic tools from harmonic analysis to state and solve the moment problem on the Sobolev space of periodic multivariate trigonometric functions. This allows us to construct an infinite-dimensional moment-SOS hierarchy to solve various kinds of {\em Sobolev POPs}, namely non-convex POPs on Sobolev spaces, with global convergence guarantees.

In order to keep this paper as short and elementary as possible, we do not describe here potential applications of the moment-SOS hierarchy for solving non-linear calculus of variations problems, or optimal control problem involving non-linear partial differential equations. Such applications are certainly very promising, and they will be reported in further communications.

The outline of the paper is as follows. In Section \ref{sec:mom} we state our Sobolev moment problem. In Section \ref{sec:fourier} we reformulate our Sobolev moment problem as a moment problem in the Fourier coefficients. In Section \ref{sec:approx} we propose inner and outer semidefinite approximations of the Sobolev moment cone. This allows us to solve different types of Sobolev POPs with an infinite-dimensional moment-SOS hierarchy in Section \ref{sec:sobolev}. Concluding remarks and potential extensions are mentioned in Sections \ref{sec:recover} and \ref{sec:conclu}.

\section{Sobolev moment problem}\label{sec:mom}

{Let $\N$ denote the set of natural numbers, including zero. Let $\Z$ denote the set of relative integers, including zero. Let $n$ and $m$ be given positive integers.}
Consider the space of Sobolev functions on the $n$-dimensional unit torus $T^n$ whose derivative up to order $m$ are square integrable:
$$H^m(T^n) := \{ f {\in L^2}: T^n \to \C : \|f\|^2_{H^m(T^n)} < \infty\}\hd{,}$$
where the norm \hd{$\| \cdot\|_ {H^m(T^n)}$} is defined as:
$$\|f\|^2_{H^m(T^n)} := \sum_{|a| \leq m} \int_{T^n} \|D^a f(x)\|^2 dx$$
with {$\|.\|$ denoting the Euclidean norm,} \hd{$a=(a_i)_{i=1,\ldots n}\in\mathbb{N}^n$}, $|a| = \sum_{i=1}^n a_i$ and $D^a = \frac{\partial^{|a|}}{\partial x_1^{a_1}\dots \partial x_n^{a_n}}$.

Let $c_0(\Z^n)$ denote the set of sequences consisting of a finite number of elements of $\Z^n$, allowing repetitions.
Let us consider the closed bounded unit ball of $H^m(T^n)$
\[
B:=\{f \in H^m(T^n) : \|f\|_{H^m(T^n)} \leq 1\}.
\]
Let $\mu$ be a measure\footnote{{In this paper, by measure we mean a non-negative Radon measure, i.e. locally finite and tight. Locally finite means that every point has a neighborhood of finite measure. Tight (or inner regular) means that the measure of any Borel set $X$ is the supremum of the measure of compact sets included in $X$.}} supported on $B$, and let $\ba \in c_0(\Z^n)$. The {\it moment} of $\mu$ of index $\ba$ is defined as:
\begin{equation}
y_{\ba} := \int_B m_{\ba}(f) ~ d\mu(f)
\end{equation}
\hd{where} $$m_{\ba}(f) := \prod_{a \in \ba} \scal{f}{e_a}_{{L^2}}$$
the {\it monomial} of $f$ of index $\ba$, for the scalar product
\begin{equation}\label{exp}
\scal{f}{e_a}_{{L^2}}:=\int_{T^n} f(x) e_a(x) dx, \quad e_a(x):=e^{-2\pi{\bf i}\scal{a}{x}_{\R^n}}.
\end{equation}
Note that\hd{,} since $\ba \in c_0(\Z^n)$, \hd{the monomial} $m_{\ba}(f)$ is the product of finitely many coefficients. The {\it algebraic degree} \hd{of $\ba$} is $d_{\ba}:=\#\ba$, \hd{where $\#\ba$ is} the cardinality of $\ba$, i.e. the number of terms in the product defining $m_{\ba}(f)$.
The {\it harmonic degree} \hd{of $\ba$} is the integer $\delta_{\ba}:=\max_{a\in\ba} \max_{i=1,\ldots,n}|a_i|$.

\begin{example}\label{exmom}
	Let $n=2$.
	The empty set $\ba=\{\} = \emptyset$ indexes the mass $\int_B d\mu(f)$, $\ba=\{(0,0)\}$ indexes the first degree moment $\int_B \scal{f}{e_{(0,0)}}_{{L^2}} d\mu(f)$, $\ba=\{(0,0),(0,0)\}$ indexes the second degree moment $\int_B \scal{f}{e_{(0,0)}}^2_{{L^2}} d\mu(f)$, $\ba=\{(1,0),(0,-1),(0,-1)\}$ indexes the third degree moment  $\int_B \scal{f}{e_{(1,0)}}_{{L^2}} \scal{f}{e_{(0,-1)}}^2_{{L^2}} d\mu(f)$, etc.
\end{example}

\hd{
	Given an index set $A {\:\in\:} c_0(\Z^n)^N$, the problem addressed in this paper is the membership oracle for the Sobolev moment cone
	\begin{equation}\label{smc}
	C(A) := \{(y_{\ba})_{\ba \in A} : y_{\ba} = \int_B m_{\ba}(f)\,d\mu(f) \:\text{for some}\: \mu\:\text{supported on}\:B\} \subset \C^N.
	\end{equation}
	
	{\bf Sobolev moment problem}: {Let $N$ be a given positive integer}. Given an index set $A \in c_0(\Z^n)^N$, does a given vector $(y_{\ba})_{\ba \in A} \in \C^N$ belong to the Sobolev moment cone $C(A)$  ?
	
	Solving the Sobolev moment problem amounts to finding a measure $\mu$ supported on $B$ such that 
	\begin{equation}\label{mp}
	y_{\ba} = \int_B m_{\ba}(f) ~ d\mu(f), \:\text{for all}\: \ba \in A.
	\end{equation}
}

Since we defined a monomial \hd{of a function $f$} in Sobolev space, we can also define a {\it polynomial} $p$ \hd{of $f$} as a linear combination of monomials:
\[
p(f) := 
\sum_{\ba \in \text{spt}\:p} p_{\ba} m_{\ba}(f)
\]
where the support \hd{of $p$, denoted by} $\text{spt}\:p \subset c_0(\Z^n)$, is the \hd{the set of all indices of the monomials of $f$ appearing in p with nonzero coefficients}, the algebraic degree \hd{of $p$} is $d(p):=\max_{\ba \in \text{spt}\:p} d_{\ba}$ and the harmonic degree \hd{of $p$} is $\delta(p):=\max_{\ba \in \text{spt}(p)}\delta_{\ba}$.

\section{Fourier embedding}\label{sec:fourier}

Let us reformulate our moment problem in the space of Fourier coefficients.
The following results are classical \cite{fournier2003}.

Define the Fourier transform $F:L^2(T^n) \to \ell_2(\Z^n), \:\: f \mapsto c$ where $c:=(c_a)_{a\in\Z^n}$ and $c_a:=\scal{f}{e_a}_{{L^2}}$ is the Fourier coefficient of index $a \in \Z^n$ of $f$\hd{, where $e_a$ was defined in \eqref{exp}.}
The adjoint of $F$ is the inverse Fourier transform $F^*:\ell_2(\Z^n) \to L^2(T^n), \:\: c \mapsto f =
\scal{c}{e}_{\ell^2} =   \sum_{a\in\Z^n} c_a e_{-a}$
where $e:=(e_a)_{a\in \Z^n}$.

\begin{proposition}
	The space $H^m(T^n)$ admits an equivalent norm in terms of the Fourier basis: \hd{for any $f \in H^m(T^n)$ and $c=(c_a)_{a\in \Z^n}$ it holds}
	$\|f\|^2_{H^m(T^n)} =  \sum_{a \in \Z^n} w_a {|c_a|}^2$, 
	$c=Ff$ and $w_a:=(1+\scal{a}{a}_{\hd{\ell_2}})^m$, $a \in \Z^n$.
\end{proposition}

Define the diagonal operator 
{$W: \ell_2(\Z^n) \to \ell_2(\Z^n), \:\: (c_a)_{a\in\Z^n} \mapsto (c_a/\sqrt{w_a})_{a\in\Z^n}$.}

{
	\begin{proposition}\label{isomorphic} 
		The map $F^*W$ gives a topological isomorphism between $\ell^2(\Z^n)$ and $H^m(T^n)$. We have
		$\|F^*Wc\|^2_{H^m(T^n)} = \sum_{a \in \Z^n} w_a (|c_a|/\sqrt{w_a})^2 = \|c\|_{\ell_2(\Z^n)}$
		for any $c \in \ell_2(\Z^n)$. This yields
		$H^m(T^n) = \{F^*Wc : c \in \ell_2(\Z^n) \}.$
	\end{proposition}
}

From this it follows that for any $f \in H^m(T^n)$ there exists a unique $c \in \ell_2(\Z^n)$ such that $f = F^* W c$. Now define the set of Fourier coefficients
$$
E := \{Ff : f \in B\} {=FB} \subset \ell_2(\Z^n).
$$

\begin{proposition}\label{ecompact}
	$E$ is compact \hd{in $\ell_2(Z^n)$}. Moreover $B$ is one to one to $E$ and in particular $B = {F^*E = } \{F^* c : c \in E\}$.
\end{proposition}

\begin{proof}
	
	$E$ is the linear image of the closed unit ball $B$ of $H^m(T^n)$ via $F$.  By \hd{Proposition \ref{isomorphic}} and the definition of $E$ we have
	{
		$ E= FB = \{F f : \|f\|_{H^m(T^n)} \leq 1 \} = \{F F^*Wc : \|F^*W c\|_{H^m(T^n)} \leq 1, c \in \ell_2(\Z^n)\} = \{Wc : \|c\|_{\ell_2(T^n)} \leq 1\}$. 
	}
	The compactness of $E$ is obtained by noting that $E = W Q$\hd{,} where $Q$ is the \hd{closed} unit ball of $\ell_2(\Z^n)$ and $W$ is a compact operator from $\ell^2(\Z^n)$ to itself, since it is a diagonal operator with diagonal $1/\sqrt{w_a} \to 0$ as $|a| \to \infty$. Indeed\hd{,} by definition\hd{,} a compact operator between normed spaces maps a bounded set in a relatively compact space (i.e. a set that has a compact closure), and here we are mapping $B$ that is the closure of the unit ball of $\ell_2(\Z^n)$.
\end{proof}

{Geometrically, $E$ is an ellipsoid contained in a Hilbert box with vanishing axes:
	$$E = \{c \in \ell_2(\Z^n): \sum_{a \in \Z^n} w_a c_a^2 < \infty \} \subset \{c \in \ell_2(\Z^n): |c_a| \leq w_a^{-1/2}, \forall~a\in \Z^n\}.$$
}

Let
\begin{equation}\label{pushforward}
\nu := {F_{\#}\mu}
\end{equation}
denote the pushforward measure of $\mu$ through \hd{$F$}. 
For any $\ba \in c_0(\Z^n)$, the moment $y_{\ba}$ of $\mu$ can then be expressed as a moment of $\nu$ in the space of Fourier coefficients:
\[
y_{\ba} = \int_B m_{\ba}(f) \, d\mu(f) = \int_{E} c^{\ba} ~ d\nu(c), \qquad c^{\ba} := \prod_{a \in \ba} c_a.
\]
\hd{The Sobolev moment problem can therefore be stated as a membership oracle in the Fourier moment cone
	\begin{equation}\label{fmc}
	\{(y_{\ba})_{\ba \in A} : y_{\ba} = \int_E c^{\ba}\,d\nu(c) \:\text{for some {measure} }\: \nu\:\text{supported on}\:E\} \subset \C^N.
	\end{equation}
	It consists of finding a measure $\nu$ supported on $E$ such that
	\begin{equation}\label{mpe}
	y_{\ba} = \int_{E} c^{\ba} ~ d\nu(c), \:\text{for all}\: \ba \in A.
	\end{equation}
}

\begin{proposition}\label{prop:mpe}
	There is a solution to the Sobolev moment problem \eqref{mp} on the ball $B$  if and only if there is a solution to the Fourier moment problem \hd{\eqref{mpe}}
	on the ellipsoid $E$.
\end{proposition}

\begin{proof}
	{
		If $\mu$ is such that \eqref{mp} holds, then $F_{\#}\mu$ fulfills \eqref{mpe} by the definition \eqref {pushforward} of pushforward measure. Conversely, suppose that $\nu$ is such that \eqref{mpe} holds and observe that $B=F^*E$ by Proposition \ref{ecompact}. Therefore, using the orthonormality of the complex exponentials $(e_a)_a$ and the change of variables formula, we get: $$\int_E \prod_a c_a d\nu=\int_E \prod_a\langle F^*c, e_a\rangle d\nu=\int_B \prod_a\langle f, e_a\rangle dF^*_{\#}\nu$$
		i.e. $F^*_{\#}\nu$ fulfills \eqref{mp}.}
	
	%  the ball $B$ is one to one to the ellipsoid $E$ and in particular $B = \{F^* c : c \in E\}$, then for any $f \in B$ there exists a $c \in E$ s.t. $$\scal{f}{e_a}_{H^m(T^n)} = \scal{F^* \hd{W}c}{e_a}_{\ell_ 2} = w_a c_a.$$ Conversely, for any $f \in B$ there exists a unique $c \in E$ such that $f = F^* W c$.
\end{proof}

\section{Semidefinite approximations of the Sobolev moment cone}\label{sec:approx}

Given an index set $A {\:\in\:} c_0(\Z^n)^N$, \hd{the Sobolev moment problem is the membership oracle in the Sobolev moment cone $C(A)$ defined in \eqref{smc}.}
Despite being convex and finite-dimensional, the cone $C(A)$
is difficult to manipulate directly. 
It must be approximated by linear sections and projections of a finite-dimensional convex cone on which optimization can be carried out efficiently, namely the semidefinite cone \hd{i.e. the set of non-negative quadratic forms.}

\subsection{Outer approximations}

We can construct outer approximations, or relaxations of $C(A)$, by projecting finite-dimensional spectrahedra (i.e. linear sections of the semidefinite cone) of increasing size. Let $\Pi_A : \ell_2(\Z^n) \to \C^N, y \hd{\:=\:(y_{\ba})_{\ba \in \Z^n}} \mapsto (y_{\ba})_{\ba\in A}$ denote the projection map onto the subspace indexed by $A$.
\hd{By Proposition \ref{prop:mpe} we can} express the Sobolev moment cone \hd{\eqref{smc}} as the Fourier moment cone \hd{\eqref{fmc}}.
Let $\C[c]$ denote the {ring} of complex polynomials \hd{in} the indeterminate {$c=(c_a)_{a\in \Z^n}$, i.e. polynomials depending on infinitely countably many variables.}
\hd{As described at the end of Section \ref{sec:mom},}
elements of $\C[c]$ can be expressed as linear combinations of monomials $p(c) = 
\sum_{\ba \in \text{spt}(p)} p_{\ba} c^{\ba}$
with algebraic degree  $d(p):=\max_{\ba \in \text{spt}(p)} d_{\ba}$ and harmonic degree $\delta(p):=\max_{\ba \in \text{spt}(p)}\delta_{\ba}$.
We are particularly interested in {\it Hermitian polynomials}, i.e. elements of $\C[c]$ with values in $\R$.
{An example of a Hermitian polynomial depending on infinitely many variables is the squared norm $\|c\|^2_W:=\sum_{a \in \Z^n} w_a |c_a|^2$ defined in Proposition \ref{isomorphic}. Its harmonic degree is infinite and its algebraic degree is two, as each monomial has the form $|c_a|^2=c_a c_{-a}$, $a \in \Z^n$.}

Given $r,\rho \in \N$,
let $\C[c]_{r,\rho}:=\{p \in \C[c] : d(p)\leq r, \delta(p)\leq \rho\}$ and define the cone of Hermitian polynomial sums of squares $$\Sigma_{r,\rho}:=\left\{ \sum_k q^*_kq_k : q_k \in \C[c]_{r,\rho}\right\}$$ and
the {truncated} quadratic module $${Q_{r,\rho}:=\left\{s_0+s_1(1-\sum_{a \in \Z^n, \delta_a \leq \rho} w_a |c_a|^2) : s_0 \in \Sigma_{r,\rho}, s_1 \in \Sigma_{r-1,\rho}\right\}\subset \C[c]_{2r,\rho}}.$$ 
Given a sequence $y=(y_{\ba})_{\ba \in \Z^n} \in \ell_2(\Z^n)$, define the linear functional $\ell_y : \C[c] \to \C,\:\: p(c):=\sum_{\ba} p_{\ba} c^{\ba} \mapsto \ell_y(p):=\sum_{\ba} p_{\ba} y_{\ba}$.
Let $d_A:= \max_{\ba \in A} d_{\ba}$ denote the algebraic degree of A, and let
$\rho_A:=\max_{\ba \in A} \delta_{\ba}$ denote the harmonic degree of A.
Finally, define the following cone
\[
C^{\text{out}}_{r,\rho}(A) := \Pi_A\left\{y \hd{\:\in\:\ell_2(\Z^n)} : \ell_y(p) \geq 0 \:\text{for all}\: p \in Q_{r,\rho}\right\}.
\]

\begin{proposition}\label{out}
	For any $r \geq d_A$ and $\rho \geq \rho_A$,
	$C^{\mathrm{out}}_{r,\rho}(A)$ is a semidefinite representable\footnote{\hd{A set is semidefinite representable if it can be expressed as a projection of a spectrahedron. A spectrahedron is an affine section of the semidefinite cone.}} outer approximation of $C(A)$.
\end{proposition}

\begin{proof}
	To prove the outer approximation claim, let us \hd{take} a vector $y \in C(A)$ and prove that $y \in C^{\text{out}}_{r,\rho}(A)$. Since $y \in C(A)$, \hd{by} Proposition \ref{prop:mpe}\hd{,} there exists a measure $\nu$ supported on $E$ such that $y_{\ba} = \int_E c^{\ba} d\nu(c)$. In particular, for any real valued polynomial $p(c) = \sum_{\ba} p_{\ba} c^{\ba}$ which is non-negative on $E$, \hd{the} vector $y$ is such that $\ell_y(p) = \sum_{\ba} p_{\ba} y_{\ba} = \sum_{\ba} p_{\ba} \int_E c^{\ba}d\nu(c) = \int_E p(c) d\nu(c)$ is nonnegative. In particular this holds for polynomials {in $Q_{r,\rho}$} and hence $y \in C^{\mathrm{out}}_{r,\rho}(A)$.
	
	To prove the semidefinite representability claim, observe that the quadratic form $\C[c] \to \R, q \mapsto \ell_y(q^*q)$ can be expressed as a Hermitian matrix linear in $y$. Non-negativity of the quadratic form is therefore equivalent to positive semidefinitess of a matrix which is linear in $y$, i.e. \hd{equivalent to} a linear matrix inequality. Testing non-negativity of 
	$\ell_y(p)$ for all $p \in Q_{r,\rho}$
	amounts to testing non-negativity of $q \mapsto \ell_y(q^*q)$ {for all $q \in \C[c]_{r,\rho}$} and $q \mapsto \ell_y({(1-\sum_{a \in \Z^n, \delta_a \leq \rho} w_a |c_a|^2)}q^*q)$ {for all $q \in \C[c]_{r-1,\rho}$}. These quadratic forms are finite dimensional, so it follows that testing membership in $C^{\mathrm{out}}_{r,\rho}(A)$ amounts to testing membership in the projection of a \hd{spectrahedral cone}, \hd{i.e.} a finite-dimensional linear slice of the semidefinite cone.
\end{proof}

\begin{proposition}\label{outdense}
	$\overline{C^{\mathrm{out}}_{\infty,\infty}}(A)=C(A)$.
\end{proposition}

\begin{proof}
	%According to Putinar's Positivstellensatz -- see e.g. \cite[Thm. 3.20]{l09} or \cite[Thm. 2.14]{l10}, every polynomial $p$ which is strictly positive on $E$ can be written as $p=s_0+s_1(1-\|c\|^2_W)$ for $s_0$ and $s_1$ sums of Hermitian squares of polynomials. So the closure of the quadratic module $Q_{r,\rho}$ coincides with the cone of polynomials that are non-negative on $E$. 
	{We use an infinite-dimensional extension of the finite-dimensional Putinar's Positivstellensatz -- \hd{see \cite{p93}} and \cite[Thm. 3.20]{l09} or \cite[Thm. 2.14]{l10}. This extension is available in any unital commutative algebra, and in particular in the algebra $A$ generated by elements of $\C[c]$. A quadratic module $Q$ is a (generally infinite-dimensional) cone such that $1 \in Q$, $Q+Q \subset Q$ and $A^2 Q \subset Q$. It is Archimedean if $\forall p \in A$, $\exists N \in \N$ such that $N - p^2 \in Q$. It can be checked that the quadratic module $Q_{\infty,\infty} := \cup_{r,\rho \in \N} Q_{r,p} \subset \C[c]$ associated to the Fourier ellipsoid $E$ is indeed Archimedean.
		In this context, \cite[Theorem 3.9]{ikkm23} states that any linear functional on $A$ which is \hd{non-negative} on $Q_{\infty,\infty}$ has a unique representing measure with compact support in $E$, i.e.
		$\{y \in \ell_2(\Z^n) : \ell_y(p) \geq 0 \text{ for all } p \in Q_{\infty,\infty}\} = \{y \in \ell_2(\Z^n): y_{\ba} = \int c^{\ba} d\nu \text { for some measure $\nu$ supported on $E$ }\}$ is the full moment cone on $E$. This result is the dual to the infinite-dimensional Jacobi's Positivstellensatz of \cite[Theorem 2.1]{gkm14}, originally stated in \cite[Theorem 4]{j01}. The proof is concluded by observing that $C(A)$ is just the (closure of the) finite-dimensional projection through \hd{the} map $\Pi_A$ of the full moment cone.
	}
\end{proof}

\begin{proposition}
	For any $r \geq r_A$, the Hausdorff distance $d_H$ between $C(A)$ and $C^{\mathrm{out}}_{r,{\rho_A}}(A)$ is bounded as follows
	\begin{equation}
	d_H(C(A), C^{\mathrm{out}}_{r,{\rho_A}}(A)) ~\leq~  9\, (2\rho_A + 1)^n\, \frac{r_A^2}{r^2}.
	\end{equation}
\end{proposition}

\begin{proof}
	Any polynomial in $\mathbb{C}[c]_{r,\rho}$ can be written as {$p(c) = \langle w,  \phi(c) \rangle_{\C^K}$, for some  $w \in \C^K$ and with $\phi: \C^N \to \C^K$}, where $N$ is the number of Fourier coefficients up to harmonic degree $\rho$ for functions on the torus $T^n$, i.e. $N = (2\rho + 1)^n$, while $K$ is the number of Chebyshev polynomials up to degree $r$ that we can build on $\C^N$, i.e. $K = \binom{r+N+1}{r}$. This representation allow us to identify $C^{\mathrm{out}}_{r,\rho}(A)$ with $\widehat{\cal K}_s$ with $s = r$ and $C(A)$ with ${\cal K}_s$ and $s = r$ from \cite{br23} and \hd{then} use their Corollary 1 (note that we do not have a scale factor $1/(2r+1)^d$).
	
	Denote by $y(\mu) \in \C^K$ the vector of all the moments of a measure $\mu$ supported on the $N$ Fourier coefficients up to algebraic degree $r$. By applying Corollary 1 of \cite{br23}, we have that for any $\widehat{y} \in C^{\mathrm{out}}_{r,\rho}$ there exists a measure $\mu$ supported on the $N$ Fourier coefficients, such that
	$$
	\|\Pi^{(r_A)}_{r}(\widehat{y} - y(\mu))\|_{\mathrm{Fro}} \leq \frac{9 N r_A^2}{r^2}, 
	$$
	where $\|\cdot\|_{\mathrm{Fro}}$ is the Frobenius norm and where $\Pi^{(r_A)}_{r}$ is a diagonal matrix such that $(\Pi^{(r_A)}_{r})_{a,a} = 1$ iff $a$ is the index of a Chebyshev polynomial with algebraic degree less or equal to $r_A$, otherwise it is $(\Pi^{(r_A)}_{r})_{a,a} = 0$. The proof is concluded by noting that $C(A) = \{ \Pi^{(r_A)}_{r} y(\mu) ~|~ \mu \in M_\rho \}$, where $M_\rho$ is the set of measures supported on Fourier coefficients with maximum harmonic degree $\rho$.
	
	{ By Proposition \ref{out} we have that $C(A) \subseteq C^{\mathrm{out}}_{r,\rho}(A)$, so the definition of Hausdorff distance with respect to the Euclidean norm simplifies to
		$$
		d_H(C(A), C^{\mathrm{out}}_{r,\rho}(A))  = \sup_{\widehat{y} \in C^{\mathrm{out}}_{r,\rho}(A)} \inf_{y \in C(A)} \|y - \widehat{y}\|_{\mathrm{Fro}}.
		$$
		Now note that,  by construction $\Pi_A y = y$ for any $y \in C(A)$ and $\Pi_A \widehat{y} = \widehat{y}$ for any $\widehat{y} \in C^{\mathrm{out}}_{r,\rho}$. Moreover, since $\Pi_A$ is a diagonal matrix with $(\Pi_A)_{a,a} = 1$ only if $a \in A$, otherwise $0$, then $\Pi_A \Pi^{(r_A)}_r = \Pi^{(r_A)}_r \Pi_A = \Pi_A$.
		Then we have 
		$$\|y - \widehat{y}\|_{\mathrm{Fro}} = \|\Pi_A(y - \widehat{y})\|_{\mathrm{Fro}} = \|\Pi_A \Pi^{(r_A)}_{r}(y - \widehat{y})\|_{\mathrm{Fro}} \leq \|\Pi_A\|_{\mathrm{op}} \|\Pi^{(r_A)}_{r}(y - \widehat{y})\|_{\mathrm{Fro}}$$
		where $\|\Pi_A\|_{\mathrm{op}}$ is the operator norm of $\Pi_A$ and satisfies $\|\Pi_A\|_{\mathrm{op}} \leq 1$ since $\Pi_A$ is a projection operator. Then
		$d_H(C(A), C^{\mathrm{out}}_{r,\rho}(A)) \leq \sup_{\widehat{y} \in C^{\mathrm{out}}_{r,\rho}(A)} \inf_{y \in C(A)} \|\Pi^{(r_A)}_r(y - \widehat{y})\|_{\mathrm{Fro}} \leq \sup_{\widehat{y} \in C^{\mathrm{out}}_{r,\rho}} \inf_{y \in C(A)} \|\Pi^{(r_A)}_r(y - \widehat{y})\|_{\mathrm{Fro}}
		\leq \frac{9 N r_A^2}{r^2}.$
	}
\end{proof}

\subsection{Inner approximations}

Another approach \hd{for approximating $C(A)$} consists of expressing \hd{the representing} measure $\nu$ as being absolutely continuous with respect to some reference measure $\gamma$ whose moments can be easily calculated, e.g. the Gaussian measure on $\ell_2(\Z^n)$ \cite{b98}.

Let
\[
C^{\text{inn}}_{r,\rho}(A) := \{(y_{\ba})_{\ba \in A} : y_{\ba} = \int_E c^{\ba} p(c) d\gamma(c) \:\text{for some}\: p \in Q_{r,\rho}\}.
\]
\begin{proposition}\label{inn}
	For all $r,\rho \in \N$,
	$C^{\mathrm{inn}}_{r,\rho}(A)$ is a semidefinite representable inner approximation of $C(A)$.
\end{proposition}

\begin{proof}
	{
		Note that $\nu = p \gamma$ is a measure. Indeed\hd{,} it is absolutely continuous with respect to $\gamma$, with Radon-Nikod\'ym derivative $p$, which is a continuous and bounded function over the compact set $E$, and, in particular, non-negative by construction, since it is  constrained to the quadratic module $Q_{r,\rho}$.
		Since the set $\{p\gamma : p \in Q_{r,\rho}\}$ is, by construction\hd{,} a subset of the measures over $E$, we have that 
		$C^{\mathrm{inn}}_{r,\rho} \subseteq C(A).$
	}
\end{proof}

\begin{proposition}\label{inndense}
	$\overline{C^{\mathrm{inn}}_{\infty,\infty}}(A)=C(A)$.
\end{proposition}

\begin{proof}
	As \hd{already mentioned} in the proof of Proposition \ref{outdense}, according to {Jacobi}'s Positivstellensatz {-- see \cite[Theorem 2.1]{gkm14}, originally stated in \cite[Theorem 4]{j01} --} elements of $Q_{r,\rho}$ can approximate as closely as desired any polynomial nonnegative on $E$, i.e. we can construct a sequence $p_{r,\rho} \in Q_{r,\rho}$ such that $\|p-\hd{p_{r,\rho}}\|_W \to 0$ when $r,\rho \to \infty$. It follows that for all $\ba \in A$, $\int_E c^{\ba} p_{r,\rho}(c) d\gamma(c) \to \int_E c^{\ba} p(c) d\gamma(c) = y_{\ba}$ when $r,\rho \to \infty$.
\end{proof}

\section{Solving Sobolev POPs}\label{sec:sobolev}

\subsection{Harmonic Sobolev POP}

We are now fully equipped to solve a harmonic Sobolev POP (polynomial optimization problem) of the form
\begin{equation}\label{hspop}
p^* := \inf_{f\in B} p(f)
\end{equation}
where $$p(f) = \sum_{\ba \in 
	A} p_{\ba} m_{\ba}(f)$$ is a given Hermitian polynomial in the indeterminate $f \in B$, of support $A:=\mathrm{spt}\:p$.
Problem (\ref{hspop}) is called 
{\it harmonic} because the harmonic degree $\delta(p)$ is finite, and the problem does not involve harmonics of degrees higher than $\delta(p)$.
Note that the infimum in (\ref{hspop}) is always attained, since it is a finite-dimensional problem and $B$ is bounded.

Harmonic Sobolev POP (\ref{hspop}) is equivalent to the linear problem
\[
p^* := \min_{\mu \in P(B)} \int_B p(f)d\mu(f)
\]
on $P(B)$, the set of probability measures on the Sobolev ball $B$.
Using the Fourier embedding, harmonic Sobolev POP (\ref{hspop}) is equivalent to the harmonic Fourier POP
\[
p^* := \min_{c \in E} p(c)
\]
and the linear problem
\[
p^* := \min_{\nu \in P(E)} \int_E p(c)d\nu(c)
\]
on $P(E)$, the set of probability measures on the Fourier ellipsoid $E$. In turn, this is equivalent to the linear problem
\begin{equation}\label{hlp}
p^* := \min_{y \in C(A)} \sum_{\ba \in A} p_{\ba} y_{\ba} \:\:\mathrm{s.t.}\:\:y_{\emptyset}=1
\end{equation}
on the cone of moments $C(A)$.

Therefore\hd{,} we can design a moment-SOS hierarchy of lower bounds
\[
p^{\mathrm{out}}_{r,\rho} := \min_{y \in C^{\mathrm{ out}}_{r,\rho}(A)} \sum_{\ba \in A} p_{\ba} y_{\ba}
\]
as well as a moment-SOS hierarchy of upper bounds
\[
p^{\mathrm{inn}}_{r,\rho} := \min_{y \in C^{\mathrm {inn}}_{r,\rho}(A)} \sum_{\ba \in A} p_{\ba} y_{\ba}
\]
for increasing algebraic\hd{,} resp. harmonic\hd{,} relaxation degrees $r \geq d(p)$, $\rho \geq \delta(p)$.

\begin{theorem}\label{thm:momsos}
	%For all $r\geq r' \geq d(p)$ and $\rho\geq\rho'\geq\delta(p)$, it holds $$p^{\mathrm{out}}_{r',\rho'} \leq p^{\mathrm{out}}_{r,\rho} \leq p^{\mathrm{out}}_{\infty,\infty} = p^* = p^{\mathrm{inn}}_{\infty,\rho} \leq p^{\mathrm{inn}}_{r,\rho} \leq p^{\mathrm{inn}}_{r',\rho'}.$$
	{For all $r\geq r' \geq d(p)$ and $\rho = \delta(p)$ it holds $$p^{\mathrm{out}}_{r',\rho} \leq p^{\mathrm{out}}_{r,\rho} \leq p^{\mathrm{out}}_{\infty,\rho} = p^* = p^{\mathrm{inn}}_{\infty,\rho} \leq p^{\mathrm{inn}}_{r,\rho} \leq p^{\mathrm{inn}}_{r',\rho}.$$}
\end{theorem}

\begin{proof}
	{The objective function contains moments of harmonic degrees up to $\delta(p)$, so it is not necessary to have moments of higher harmonic degrees in the outer approximation
		$C^{\text{out}}_{r,\rho}(A)$ of the cone of moments, i.e. $\rho = \delta(p)$.} The statement then \hd{readily follows} by applying Propositions \ref{out}, \ref{outdense}, \ref{inn} and \ref{inndense}.
\end{proof}

\subsubsection{Example}

Consider the harmonic Sobolev POP
\[
p^*=\min_{f\in B} \:\: \scal{f}{e_0}^4_{H^0(T)} + (\scal{f}{e_1}^2_{H^0(T)}-1/4)^2
\]
on $B \subset H^0(T)$, i.e.
$n=1$, $m=0$ and harmonic degree $\rho=1$. Observe that the function to be minimized is non-convex in $f$.

The harmonic Sobolev POP is equivalent to the harmonic Fourier POP
\[
p^*=\min_{c_{-1},c_0,c_1} \:\: c^4_0 + (c^2_1-1/4)^2 \:\mathrm{s.t.}\:c^2_{-1}+c^2_0+c^2_1 \leq 1.
\]
Note that the Fourier coefficient $c_{-1}$ does not appear in the objective function, and hence without loss of generality it can be set to zero.

With the outer moment-SOS hierarchy, at algebraic relaxation degree $r=2$, we obtain  the two global minimizers $c^*_{-1}=0$, $c^*_0=0$, $c^*_1=\pm 1/2$ and the corresponding functions $f^*(x)=\pm e^{-2\pi{\bf i}x}/2$ achieving the global minimum $p^*=p^{\mathrm{out}}_{2,1}=0$.

\subsection{Algebraic Sobolev POP}

Another class of POP on Sobolev functions is 
\begin{equation}\label{aspop}
p^* = \inf_{f \in B} L(p(f,D^{a_1}f,\ldots,D^{a_l}f))
\end{equation}
where $p$ is a given real valued multivariate polynomial of degree $d_p$ of a function $f \in B$ and its derivatives $D^{a_j}f$, $a_j \in \N^n$, $j=1,\ldots,l$ and $L : L^\infty(T^n) \to \R$ is a given {bounded} linear functional. \hd{The} coefficients of $p$ are bounded functions \hd{on $T^n$}. For example
\begin{equation}\label{sp}
L(p(f))=\int_{T^n} (p_1(x)f(x)+p_2(x)\|Df(x)\|^2_2) d\sigma(x)
\end{equation}
where $\sigma$ is a given probability measure on $T^n$ and $p_1, p_2$ are given real polynomials of \hd{the variable} $x$.

Note that\hd{,} contrary to harmonic problem (\ref{hspop}), the non-linearity hits directly the function value $f(x)$ and its derivatives, and hence problem (\ref{aspop}) generally involves infinitely many harmonics. Still, problem (\ref{aspop}) is called {\it algebraic} because $p$ is a finite degree  polynomial.

The following result guarantees that the non-linear functional defined above is well defined on $H^m(T^n)$ with $m$ large enough.

\begin{proposition}
	{The functional of Sobolev POP \eqref{aspop} is bounded when $f \in H^{s + n/2+1}_2(T^n)$ where $s = \max_{j=1,\ldots,\hd{l}} |a_j|$.}
\end{proposition}
\begin{proof}
	By the Sobolev embedding theorem \cite{fournier2003} {in a bounded set $\Omega$ with Lipschitz boundary in $\R^n$}, when $f \in H^{m+n/2+1}_2(\Omega)$ then $D^a f$ is a Lipschitz function for any $a$ satisfying $|a| \leq m$. Since now 
	$D^{a_j} f \in L^q(\Omega)$ for any $q \in [1,\infty]$ (due to the boundedness of $\Omega$), the desired result is obtained by applying the H\"older inequality.$\:$`
\end{proof}

Let us express the objective function of (\ref{aspop}) as a polynomial function of $c$, the Fourier coefficients of $f$. Indeed, if $f = \scal{c}{e}_{\ell_2} = \sum_{a \in \Z^n} c_a e_{-a}$ then a monomial of degree $d \in \N$ writes $$f^d=\scal{c}{e}^d_{\ell_2} = \sum_{a_1, a_2, \ldots, a_d \in \Z^n} c_{a_1} c_{a_2} \cdots c_{a_d} e_{-(a_1+a_2+\ldots+a_d)}$$
and it follows that 
$$L(f^d)=\sum_{a_1, a_2, \ldots, a_d \in \Z^n} 
c_{a_1} c_{a_2} \cdots c_{a_d}
z_{-(a_1+a_2+\ldots+a_d)} $$ where
\begin{equation}\label{mlf}
z_a:=L(e_a)
\end{equation}
is the moment of index $a \in \Z^n$ of \hd{the} linear functional $L$. 
Similarly, successive derivatives of $f$ will be expressed as linear functions of $c$, and hence polynomials of these derivatives will be multivariate polynomials of $c$.
Overall, the objective function is a polynomial $q$ in infinitely countably many variables with finite algebraic degree $d(q)=d_p$ and infinite harmonic degree $\delta(q)=\infty$.
Algebraic Sobolev POP (\ref{aspop}) can therefore be written equivalently as the algebraic Fourier POP
\[
p^*=\inf_{c \in E} q(c) = \sum_{\ba \in \Z^n} q_{\ba} c^{\ba}.
\]
In order to apply the moment-SOS hierarchy, we reformulate this POP as a
linear problem
\[
p^* := \min_{\nu \in P(E)} \int_E q(c)d\nu(c)
\]
on $P(E)$. In turn, this is equivalent to the infinite-dimensional linear problem
\[ 
p^* := \min_{y \in C(\Z^n)} \sum_{\ba \in \Z^n} q_{\ba} y_{\ba} \:\:\mathrm{s.t.}\:\:y_{\emptyset}=1
\]
on the full cone of moments
\[
C(\Z^n)  := \left\{(y_{\ba})_{\ba \in \Z^n} : y_{\ba} = \int_E c^{\ba}\,d\nu(c) \:\text{for some}\: \nu\:\text{supported on}\:E\right\} \hd{\:\subset\:} \ell_2(\Z^n).
\]
In contrast, the harmonic Sobolev POP of the previous section was reformulated as the finite-dimensional linear problem \eqref{hlp} on a truncated cone of moments.

As in the previous section, for every finite algebraic \hd{degree $r$ and} harmonic degree $\rho$, we can define the outer approximations
\[
C^{\text{out}}_{r,\rho}(\Z^n) := \Pi_A\{(y_{\ba})_{\ba \in \Z^n} : \ell_y(q)= \sum_{\ba \in \Z^n} q_{\ba}c^{\ba} \geq 0 \:\text{for all}\: q \in Q_{r,\rho}\}
\]
and inner approximations
\[
C^{\text{inn}}_{r,\rho}(\Z^n) := \{(y_{\ba})_{\ba \in \Z^n} : y_{\ba} = \int_E c^{\ba} q(c) d\gamma(c) \:\text{for some}\: q \in Q_{r,\rho}\}
\]
such that $C^{\text{inn}}_{r,\rho}(\Z^n) \subset C(\Z^n) \subset C^{\text{out}}_{r,\rho}(\Z^n)$ and asymptotically
$\overline{C^{\text{inn}}_{\infty,\infty}(\Z^n)} = C(\Z^n) = \overline{C^{\text{out}}_{\infty,\infty}(\Z^n)}$, but these are now infinite-dimensional cones that must be truncated to be manipulated numerically.

Given algebraic \hd{degree $r$} and \hd{harmonic degree} $\rho$, let us consider the finite dimensional linear problem
\[
p^*_{r,\rho} := \min_{y \in C(A_{r,\rho})} \sum_{\ba \in A_{r,\rho}} q_{\ba} y_{\ba} \:\:\mathrm{s.t.}\:\:y_{\emptyset}=1
\]
on the finite-dimensional cone of moments $C(A_{r,\rho})$ indexed by $$A_{r,\rho}:=\{\ba \in \Z^n : d_{\ba} \leq r, \: \delta_{\ba} \leq \rho\}.$$
We can design an outer moment-SOS hierarchy
\[
p^{\mathrm{out}}_{r,\rho} := \min_{y \in C^{\mathrm{ out}}_{r,\rho}(A_{r,\rho})} \sum_{\ba \in A_{r,\rho}} q_{\ba} y_{\ba}
\]
as well as an inner moment-SOS hierarchy
\[
p^{\mathrm{inn}}_{r,\rho} := \min_{y \in C^{\mathrm {inn}}_{r,\rho}(A_{r,\rho})} \sum_{\ba \in A_{r,\rho}} q_{\ba} y_{\ba}
\]
for increasing algebraic \hd{relaxation degree $r$ and} harmonic relaxation degree $\rho$.
Our convergence result then \hd{immediately follows} from the above considerations.

\begin{theorem}
	%For finite $r, \rho$ it holds $p^{\mathrm{out}}_{r,\rho} \leq p^*_{r,\rho} \leq  p^{\mathrm{inn}}_{r,\rho}$.
	{It holds $p^{\mathrm{out}}_{\infty,\infty} = p^*$ and $p^{\mathrm{inn}}_{\infty,\infty}=p^*$.}
\end{theorem}

{
	\begin{remark}
		Note that the values $p^{\mathrm{out}}_{r,\rho}$ resp.  $p^{\mathrm{inn}}_{r,\rho}$ are not necessarily lower resp. outer bounds on the value $p^*$ since the objective functions in our moment-SOS hierarchy are truncated to finite harmonic degree, and the sign of the remainders is not known a priori.
	\end{remark}
}

\subsubsection{Example}\label{sec:example-algebraic-pop}

Consider the algebraic Sobolev POP
$$p^* = \inf_{f \in B} \int_T (f(x)^2-1/2)^2 d\sigma(x)$$ where $\sigma$ is the Dirac measure at 0 on $B \subset H^0(T)$, i.e. $m=0$ and $n=1$.
Observe that the function to be minimized is non-convex in $f$.

Since $f(0)=\sum_{a \in \Z} c_a$, the moments \hd{\eqref{mlf}} of the linear functional in the objective function are equal to one\hd{, i.e. $z_a=L(e_a)=e_a(0)=1$} for all $a \in \Z$, so the problem can be written as the algebraic Fourier POP
$$p^* = \inf_{c \in E} q(c)$$ with
$$q(c)=
\frac{1}{4} -  \sum_{a_1,a_2 \in \Z} c_{a_1} c_{a_2} + \sum_{a_1,a_2,a_3,a_4 \in \Z} c_{a_1} c_{a_2} c_{a_3} c_{a_4}.$$
With the outer moment-SOS hierarchy, at algebraic relaxation degree $r=2$ and harmonic relaxation degree $\rho=0$, we obtain  the two global minimizers $c^*_0=\pm \sqrt{2}/2$
and the corresponding functions $f^*(x)=\pm \sqrt{2}/2$
achieving the global minimum $p^*=p^{\mathrm{out}}_{2,0}=0$.

%Note that for such problems it may be desirable to penalize the higher degree Fourier coefficients with a quadratic regularization term.

\subsection{{Kernel Sobolev POP}}\label{sec:kernel-pop}

While {an algebraic Sobolev POP generally} requires an infinite number of Fourier coefficients to be expressed, there exists a better basis, based on {\em kernel methods} \cite{aronszajn1950theory}, where the problem admits a representation in terms of a finite number of coefficients. Since $H^m(T^n)$ is a {\em reproducing kernel Hilbert space} when $m > n/2$, there exists a kernel function $k: T^n \times T^n \to \R$ such that $k(x,y) = k(y,x)$, $k(\cdot, x) \in H^m(T^n)$ for any $x,y \in T^n$ and more importantly, we have the {reproducing property}: for any $f \in H^m(T^n)$ and any $x \in T^n$, the following holds
$$f(x) = \langle f, k(\cdot, x) \rangle_{H^m(T^n)}.$$
In particular, for the case of $H^m(T^n)$ the kernel is known in closed form in terms of the Bessel function of the second kind, {see \cite[Sec. 7.4]{berlinet2011reproducing}}.
Then the powerful and fundamental result in machine learning known as the {Representer Theorem} \cite{scholkopf2001generalized} holds. 
\begin{theorem}\label{thm:kernel-pop}
	The Sobolev POP
	\begin{equation}\label{kpop}
	\min_{f \in H^m(T^n)} p(f(x_1),\dots, f(x_l))
	\end{equation}
	for a given polynomial $p$ is equivalent to the finite-dimensional POP
	\begin{equation}\label{fdpop}
	\min_{w \in  \R^l} p({\langle c_1, w \rangle_{\R^l}},\dots, {\langle c_n, w \rangle_{\R^l}})
	\end{equation}
	where $c_j := (k(x_i, x_j))_{i=1,\ldots,l}$, $j=1,\ldots,l$. The first problem admits a solution \hd{if} and only if the second problem admits a solution, and both problems have the same value. In particular, denoting by $f^*$ the solution of the first problem and $w^*$ the solution of the second problem, we have
	$$
	f^*(\cdot) = \sum_{j=1}^l w^*_j k(\cdot, x_j).
	$$
	More generally, the Sobolev POP
	\begin{equation}\label{gkpop}\min_{f \in H^m(T^n)} p(\langle g_1, f \rangle_{H^m(T^n)}, \dots, \langle g_l, f \rangle_{H^m(T^n)})\end{equation}
	for a given polynomial $p$ and given $g_j \in H^{-m}(T^n)$ is equivalent to the finite-dimen\-sional POP \eqref{fdpop}
	where $c_j = (\langle g_i, g_j \rangle_{H^m(T^n)})_{i=1,\ldots,l}$,
	$j=1,\ldots,l$ and $$f^* = \sum_{i=1}^l w^*_j g_j.$$ 
\end{theorem}

Note that POP \eqref{kpop} is a particular case of POP \eqref{gkpop} corresponding to the choice $g_j(.) = k(.,x_j)$  since $\langle f, k(.,x_j) \rangle_{H^m(T^n)} = f(x_j)$, $j=1,\ldots,l$.

Theorem \ref{thm:kernel-pop} implies that for kernel Sobolev POPs of the form \eqref{gkpop}, we can apply the standard finite-dimensional moment-SOS hierarchy  \cite{l10} with convergence guarantees.

Theorem \ref{thm:kernel-pop} also holds when $p$ is any continuous function which is bounded below (not necessarily a polynomial),  for any measurable space $X$ beyond $T^n$, and any space of functions on $X$ that is a reproducing kernel Hilbert space, for example any Sobolev space $H^m(X)$ where $X \subseteq \R^n$ is a domain with locally Lipschitz boundary and $m > n/2$. 

\subsubsection{Example}\label{sec:example-kernel-pop}

Revisiting Ex. \ref{sec:example-algebraic-pop}, since the objective function is $(f(0)^2-1/2)^2$, i.e. $p(t)=(t^2-1/2)^2$ and $x_1=0$, $l=1$ in Sobolev POP \eqref{kpop}, it can be expressed equivalently as the univariate POP $\min_{w_1 \in \R} \, (w^2_1 k(0,0)^2 - 1/2)^2$
whose solutions are $w_1 = \pm \frac{\sqrt{2}}{2 k(0,0)}$, corresponding to the following minimizers $f^*(x)=\pm\frac{\sqrt{2}k(x,0)}{2k(0,0)}$. 

\section{Solution recovery}\label{sec:recover}

When solving infinite-dimensional calculus of variations \hd{or optimal} control problems,
we may be faced with truncated moment problems on Sobolev spaces with an increasing number of Fourier coefficients. 
When the number of Fourier coefficients goes to infinity, we know that there is a single representing measure, i.e. the infinite-dimensional moment problem is determinate.

\begin{proposition}
	A \hd{measure} $\mu$ supported on $B$ is uniquely determined by its infinite sequence of moments $(y_{\ba})_{\ba \in c_0(\Z^n)}$.
\end{proposition}
\begin{proof}
	{
		Let $\nu$ be the pushforward measure via the Fourier transform $F$ associated to $\mu$ as in \eqref{pushforward}. Note that $\mu$ and $\nu$ are in one-to-one relation via the invertible linear map.
		\hd{The sequence $(y_{\ba})_{\ba \in c_0(\Z^n)}$ of moments of $\mu$ is also the unique sequence of moments of $\nu$. Indeed, $E$ is compact and Hausdorff and $P = \mathrm{span}\{c^{\ba}\}_{\ba \in c_0(\Z^n)}$ is a sublagebra of $C(E, \R)$ separating points. Then, by the Stone-Weierstrass theorem, $P$ is dense in $C(E,\R)$, i.e. for any
			%The sequence of moments $(y_{\ba})_{\ba \in c_0(\Z^n)}$ in \eqref{mpe} exists and is unique with respect to $\nu$. Indeed the space $E$ is compact Hausdorff and the set $P = \mathrm{span}\{c^{\ba}\}_{\ba \in c_0(\Z^n)}$ is an algebra and is dense in the set $C(E, \R)$, by the Stone-Weierstrass theorem, for any 
			$f \in C(E, \R)$,} there exists a sequence in $P$ that converges uniformly to $f$ on $E$.
		Now assume that there exist two measures $\nu, \nu'$ that lead to the same vector $y$, then
		for any function $f$, we will have a sequence $p_n \in P$ such that $\|p_n - f\|_{C(E,\R)} \to 0$ as $n \to \infty$.
		Now note that 
		$$\nu'(p) := \int p(c) d\nu'(c) = \int p(c) d\nu(c) =: \nu(p),$$
		for any $p \in P$, since $p$ is a polynomial, i.e. a finite linear combination of monomials, and the two measures $\nu, \nu'$ have the same $y$, i.e. the same moments. So we have
		
		\begin{align*}
		|\nu(f) - \nu'(f)| & \leq  |\nu(f) - \nu(p_n)| + |\nu(p_n) - \nu'(p_n)| + |\nu'(f) - \nu'(p_n)| \\
		& =   |\nu(f) - \nu(p_n)| + |\nu'(f) - \nu'(p_n)| \\
		& \leq (\nu(E) + \nu'(E)) \|f - p_n\|_{C(E,\R)}.
		\end{align*}
		The last step is due to the fact that $E$ is compact and $f$ and $p_n$ are continuous on $E$.
		Then for any $f \in C(E, \R)$
		$$|\nu(f) - \nu'(f)| = \lim_{n \to \infty} (\nu(E) + \nu'(E)) \|f - p_n\|_{C(E,\R)} = 0.$$
	}
\end{proof}

Given a sequence of moments, we may want to recover the representing measure on $B$. 
In the finite-dimensional case, the Christoffel-Darboux kernel can be used to approximate the support of a measure given its moments \cite{lpp22}.
It would be interesting to extend this kernel to Sobolev spaces.

\section{Conclusion}\label{sec:conclu}

{In this paper we address and \hd{numerically solve} the moment problem for measures supported on the unit ball of a Sobolev space. We describe how the finite-dimensional moment-SOS hierarchy can be extended to this infinite-dimensional setup, allowing to \hd{numerically solve} polynomial optimization problems on Sobolev spaces while preserving approximation and convergence guarantees.}

All our developments are done for a specific basis of complex exponentials \eqref{exp},
but similar results could be achieved for any basis with good approximation properties for the Sobolev space $H^m(T^n)$ or other reproducing kernel Hilbert spaces, as highlighted in Section~\ref{sec:kernel-pop}.

Our approach can also be generalized with \hd{exactly the} same construction to other spaces like  Sobolev spaces on general domains, Besov or Triebel-Lizorkin spaces and more generally quasi-Banach spaces where there exists a Schauder basis with reasonable approximation properties.

Finally, applications of these techniques and the infinite-dimensional moment-SOS hierarchy to the approximation of solutions of nonlinear calculus of variations problems or optimal control involving  non-linear partial differential equations remain to be investigated.

{\section*{Acknowledgments}
	
	We are grateful to Maria Infusino, Victor Magron \hd{as well as two anonymous reviewers} for useful feedback.}

\bibliographystyle{siamplain}

\end{document}